\documentclass{amsart}
\usepackage{graphicx}
\usepackage{amssymb}

\newtheorem{theorem}{Theorem}[section]
\newtheorem{corollary}[theorem]{Corollary}
\newtheorem{lemma}[theorem]{Lemma}
\newtheorem{proposition}[theorem]{Proposition}
\theoremstyle{definition}

\theoremstyle{remark}
\newtheorem{remark}[theorem]{Remark}

\def\sign{\operatorname{sgn}}

\begin{document}

\title[Sharpening the $L^p$ triangle inequality]{Sharpening the triangle inequality: envelopes between $L^{2}$ and $L^{p}$ spaces}

\author{Paata Ivanisvili}
\address{Department of Mathematics, UC Irvine}
\email{\tt pivanisv@uci.edu}

\author{Connor Mooney}
\address{Department of Mathematics, UC Irvine}
\email{\tt mooneycr@math.uci.edu}



\begin{abstract} 
Motivated by the inequality $\|f+g\|_{2}^{2} \leq  \|f\|_{2}^{2}+2\|fg\|_{1}+\|g\|^{2}_{2}$, Carbery (2006) raised the question what is the ``right" analogue of this estimate in $L^{p}$ for $p \neq 2$. 
Carlen, Frank, Ivanisvili and Lieb (2018) recently obtained an $L^{p}$ version of this inequality by providing upper bounds for $\|f+g\|_{p}^{p}$ in terms of the quantities $\|f\|_{p}^{p}, \|g\|_{p}^{p}$ and $\|fg\|_{p/2}^{p/2}$ when $p \in(0,1] \cup [2,\infty)$, and lower bounds when $p \in (-\infty,0) \cup (1,2)$, thereby proving (and improving) the suggested possible inequalities of Carbery. We continue investigation in this direction by refining the estimates of Carlen, Frank, Ivanisvili and Lieb. We obtain upper bounds for $\|f + g\|_p^p$ also when $p \in (-\infty,0) \cup (1,2)$ and lower bounds when $p \in (0,1] \cup [2,\infty)$. For $p \in [1,2]$ we extend our upper bounds to any finite number of functions. In addition, we show that all our upper and lower bounds of $\|f+g\|_{p}^{p}$ for $p \in \mathbb{R}$, $p\neq 0$, are the best possible in terms of the quantities $\|f\|_{p}^{p}, \|g\|_{p}^{p}$ and $\|fg\|_{p/2}^{p/2}$,
and we characterize the equality cases.  
\end{abstract}


\maketitle 

\section{Introduction}
For any real-valued functions $f,g \in L^{p}$ on an arbitrary measure space, and any $p\geq 1$, one has the inequality 
\begin{align}\label{conv1}
\| f+g\|_{p}^{p} \leq 2^{p-1} \left( \|f\|^{p}_{p} + \|g\|^{p}_{p} \right).
\end{align}
The estimate (\ref{conv1}) follows from the fact that  the  map $x \mapsto |x|^{p}$ is convex. If $f=g$ in (\ref{conv1}) then the constant $2^{p-1}$ is sharp and the inequality becomes equality. On the other hand, if $f$ and $g$ have disjoint supports then the constant $2^{p-1}$ is not needed. We remark that the estimate (\ref{conv1}) reflects the convexity of the unit
ball in $L^p$, which is equivalent to the usual $L^p$ triangle (Minkowski) inequality (see e.g. \cite{CFIL}). 

In \cite{Carb1}, Carbery asked under what conditions on the sequence of functions $\{f_{j}\}\subset L^{p}$ the inequality $\sum \| f_{j} \|_{p}^{p} <\infty$ would imply  $\sum f_{j} \in L^{p}$.  If we try to adapt the inequality (\ref{conv1}) to say $n$ number of functions $f_{1}, f_{2}, \ldots, f_{n}$ instead of two, then the constant $2^{p-1}$ should be replaced by $n^{p-1}$ which grows with $n$. To remove dependence  on $n$ Carbery suggested several extensions of inequality (\ref{conv1}) which were motivated by the estimate $\|f+g\|_{2}^{2} \leq  \|f\|_{2}^{2}+2\|fg\|_{1}+\|g\|^{2}_{2}$. All of them involve the extra parameter $\|fg\|_{p/2}^{p/2}$, which measures the ``overlap'' between the functions, and the strongest one in case of two functions he could prove only for indicator functions of sets. Recently a sharpened form of the triangle inequality was obtained \cite{CFIL} which implied the proposed estimates of Carbery's. Namely, take any $p \in \mathbb{R} \setminus\{0\}$, and put
\begin{align*}
\|f\|_{p} := \left(\int_{X} |f|^{p} d\mu\right)^{1/p}  \quad \text{and} \quad \Gamma_{p} := \frac{2\|fg\|_{p/2}^{p/2}}{\|f\|_{p}^{p}+\|g\|_{p}^{p}}.
\end{align*}
Then
\begin{align}\label{carlen}
\|f+g\|^{p}_{p} \leq \left(1+\Gamma_{p}^{2/p} \right)^{p-1} (\|f\|^{p}_{p}+\|g\|^{p}_{p})
\end{align}
holds true if $p\in (0, 1]\cup [2,\infty)$, and the inequality reverses if $p \in (-\infty, 0)\cup (1,2)$, where in the latter case we assume that $f,g$ are positive almost everywhere. Since by Cauchy--Schwarz $\Gamma_{p} \in [0,1]$ for all $p \in \mathbb{R}\setminus\{0\}$ we see that (\ref{carlen}) improves on the trivial bound (\ref{conv1}). 

In this paper we continue investigation in this direction and we address the following questions:
\begin{itemize}
\item[1.] Can one further sharpen the right hand side of the estimate (\ref{carlen}) if we are allowed to use only the quantities $\|f\|_{p}, \|g\|_{p}, \|fg\|_{p/2}$?
\item[2.] What is the optimal upper bound on $\|f+g\|_{p}^{p}$ in terms of the quantities $\|f\|_{p}, \|g\|_{p}, \|fg\|_{p/2}$, also when $p \in (-\infty, 0)\cup (1,2)$? 
The same question about lower bounds on $\|f+g\|_{p}^{p}$, also when $p\in (0, 1]\cup [2,\infty)$. 
\item[3.] Can one extend these estimates to many functions, more than 2? 
\end{itemize}
We will give complete answers to Questions $1$ and $2$, and we will provide an answer to Question $3$ when $p > 0.$ 
In particular we show that for $p \in [1,2]$, if $\sum_j \|f_j\|_p^p < \infty$ and $\displaystyle\sum_{i<j} \|f_j f_j\|_{p/2}^{p/2} < \infty$, then $\sum_j f_j \in L^p$.


\section{Main results}
 Let $(X, \mathcal{A}, \mu)$ be an arbitrary measure space. In what follows we consider functions $f,\,g$ on $X$ that are measurable and nonnegative. Given $p  \in \mathbb{R} \setminus\{0\}$ we will be always assuming that $\|f\|_{p}^{p},\, \|g\|_p^p < \infty$. When $p < 0$ we allow $f,\,g$ to take the value $+\infty$, where we understand $f^p,\,g^p = 0$.

\begin{theorem}\label{teor1}
For any $p\in (0,1]\cup [2,\infty)$, and any nonnegative  $f,g$ on any measure space  we have 
\begin{align}\label{pir1}
\|f+g\|^{p}_{p} \leq \left( \left(\frac{1+\sqrt{1-\Gamma_{p}^{2}}}{2}\right)^{1/p}+\left(\frac{1-\sqrt{1-\Gamma_{p}^{2}}}{2}\right)^{1/p}\right)^{p} (\|f\|^{p}_{p}+\|g\|^{p}_{p}).
\end{align} 
 The inequality reverses if $p \in (-\infty, 0)\cup [1,2]$. Equality holds if  $(fg)^{p/2} = k (f^{p}+g^{p})$ for some constant $k \in \left[0,\frac{1}{2}\right]$. 
\end{theorem}

\begin{remark}\label{bestp}
The right hand side of  (\ref{pir1}) is the best possible in the following sense: consider the measure space  $([0,1], \mathcal{B}, dx)$. Pick any nonnegative numbers $x,y$ and $z$ such that $0 \leq z \leq \sqrt{xy}$. Then, for any $p\in (0,1]\cup [2,\infty)$ the supremum of the left hand side of (\ref{pir1}) over all nonnegative $f,g$ with fixed  $\|f\|^{p}_{p}=x, \|g\|^{p}_{p}=y, \|fg\|^{p/2}_{p/2}=z$ coincides with the right hand side of (\ref{pir1}). Similarly, for any $p \in (-\infty, 0)\cup [1,2]$ the infimum of the left hand side of (\ref{pir1}) over all such $f,g$ coincides with the right hand side of (\ref{pir1}). We justify this remark in Section \ref{Reductions}.
\end{remark}
Remark \ref{bestp} implies in particular that Theorem~\ref{teor1} refines the estimate (\ref{carlen}).
As a consequence we have the following peculiar estimate:
\begin{corollary}
For any $p\in (0,1]\cup [2,\infty)$, and any number  $\Gamma \in [0,1]$  we have 
\begin{equation}\label{Improve1}
\left( \left(\frac{1+\sqrt{1-\Gamma^{2}}}{2}\right)^{1/p}+\left(\frac{1-\sqrt{1-\Gamma^{2}}}{2}\right)^{1/p}\right)^{p} \leq \left(1+\Gamma^{2/p} \right)^{p-1}.
\end{equation}
The inequality reverses if $p \in (-\infty, 0)\cup [1,2]$. 
\end{corollary}
\noindent If we set $\Gamma := 2\frac{(ab)^{p/2}}{a^p + b^p}$ for nonnegative $a,\,b$, then after a short computation inequality (\ref{Improve1}) becomes
\begin{align}\label{ut33}
\frac{(a + b)^p}{a^p + b^p} \leq \left(1+\left(2\frac{(ab)^{p/2}}{a^p + b^p}\right)^{2/p} \right)^{p-1}.
\end{align}
The above inequality is Theorem $1.3$ from \cite{CFIL}, with $\alpha := \frac{a}{a+b}.$ 

We should mention that estimate (\ref{ut33}) does not follow solely from Theorem~\ref{teor1}. It follows from the fact that both inequalities (\ref{pir1}) and (\ref{carlen}) hold true and the fact that (\ref{pir1}) is sharp in a sense of Remark \ref{bestp}. On the other hand, by comparing the right hand sides of (\ref{pir1}) and (\ref{carlen})  one arrives at (\ref{ut33}) which coincides with Theorem~1.3 in \cite{CFIL} where it is also proved that (\ref{ut33}) implies (\ref{carlen}). 

\begin{remark}
If we let $q:=1/p$ and $x = \sqrt{1-\Gamma^{2}}$, then inequality (\ref{Improve1}) can also be written as the following ``two-point inequality:''
\begin{equation}\label{TwoPoint}
\frac{(1+x)^{q} + (1-x)^{q}}{2} \leq \left(\frac{1+(1-x^{2})^{q}}{2}\right)^{1-q}
\end{equation}
for all $q  \in \left(-\infty,\frac{1}{2}\right]\cup[1,\infty) , \; x \in [0,1]$, and the inequality reverses if $q \in \left[\frac{1}{2},1\right)$. For each fixed $q \geq 2$, inequality (\ref{TwoPoint})
improves inequality ($1.7$) from \cite{BCL} (the Gross two-point inequality) for $X = 1$ and $Y$ close to $0$, using the notation in \cite{BCL}.
\end{remark}

\vskip0.5cm

Next, let $p \in \mathbb{R}\setminus\{0\}$, and set\footnote{If $\|fg\|_{p/2}^{p/2}=0$ then we set $C_{p}=1$.}
\begin{align*}
C_{p} :=  \frac{\min\{ \|f\|_{p}^{p}, \|g\|_{p}^{p}, \|fg\|_{p/2}^{p/2}\}}{\|fg\|_{p/2}^{p/2}}. 
\end{align*}
 
\begin{theorem}\label{teor2}
For any $p\in (1,\,2)$ and any nonnegative  $f,\,g$ on any measure space  we have 
\begin{align}\label{pir2}
\|f+g\|^{p}_{p} \leq  \|f\|_{p}^{p} + \|g\|_{p}^{p} +  \left( (C_{p}^{-1/p}+C_{p}^{1/p})^{p} - C^{-1}_{p} - C_{p}\right) \|fg\|_{p/2}^{p/2}.
\end{align} 
The inequality reverses if $p \in (0,1] \cup [2,\infty)$. Equality holds in (\ref{pir2}) if one of the following three conditions holds: $f = g$ on $\{fg > 0\}$, $g = \lambda f$ on $\{f > 0\}$ for some $\lambda \geq 1$, or $f = \lambda g$ on $\{g > 0\}$ for some $\lambda \geq 1$.   

For $p \in (-\infty,\,0)$ we have 
\begin{align}\label{pir3}
\|f+g\|^{p}_{p} \leq  \left(C_{p}^{-1/p}+C_{p}^{1/p}\right)^{p} \|fg\|_{p/2}^{p/2}.
\end{align}
Equality holds in (\ref{pir3}) if one of the following three conditions holds: $f = g$ on $\{fg < \infty\}$, $g = \lambda f$ on $\{f < \infty\}$ for some $\lambda \leq 1$, or $f = \lambda g$ on $\{g < \infty\}$ for some $\lambda \leq 1$.   

\end{theorem}
\noindent Exactly the same remark as before applies to Theorem~\ref{teor2}; that is, the right hand sides of (\ref{pir2}) and (\ref{pir3}) are the best possible.
Together, Theorems \ref{teor1} and \ref{teor2}, along with the remarks about optimality, answer Questions $1$ and $2$.

Finally, we state a partial answer to Question $3$ in the case $p > 0$.
\begin{corollary}\label{cor2}
For any $p \in  [1,2]$, and any sequence of nonnegative functions $\{f_{j}\}_{j \geq 1}$ 
we have 
\begin{align*}
\| \sum_{j} f_{j}\|_{p}^{p} \leq \sum_{j}\|f_{j}\|_{p}^{p} + (2^{p}-2)\sum_{i<j}\|f_{i}f_{j}\|_{p/2}^{p/2}.
\end{align*} 
If $p \in (0,1] \cup [2, \infty)$ the inequality reverses. Equality holds if and only if
\begin{align*}
\left( \sum_{j} f_{j} \right)^{p} = \sum_{j} f_{j}^{p} + (2^{p}-2) \sum_{i<j} (f_{j}f_{j})^{p/2}
\end{align*}
almost everywhere.
\end{corollary}
\noindent In particular, when $p \in [1,\,2]$ we have that $\sum_j f_j \in L^p$ provided $\sum_j \|f_j\|_p^p < \infty$ and $\displaystyle\sum_{i<j} \|f_j f_j\|_{p/2}^{p/2} < \infty$.

The rest of the paper is organized as follows. In Section \ref{Reductions} we reduce the proofs of Theorems \ref{teor1} and \ref{teor2}, as well as the remarks about their optimality, 
to computing the concave and convex envelopes of a certain function defined on the boundary of a convex cone in $\mathbb{R}^3$. In Section \ref{EnvelopeComputations} we compute
these envelopes. Finally, in Section \ref{Corollary2} we prove Corollary \ref{cor2} using an observation about the proof of Theorem \ref{teor2}.


\section{Reductions}\label{Reductions}
In this section we reduce Theorems \ref{teor1} and \ref{teor2} to computing explicitly the convex and concave envelopes of a certain function defined on the boundary of a convex cone in $\mathbb{R}^3$. Let
$$\Omega := \{x,\,y \geq 0, \, 0 \leq z \leq \sqrt{xy}\}$$
be the convex cone in $\mathbb{R}^3$ whose vertical cross-sections $\Omega \cap \{x + y = c > 0\}$ are half-ellipses. For $p \in \mathbb{R} \backslash \{0\}$ define $\varphi_p$ on $\partial \Omega$ by
$$\varphi_p(x,\,y,\,\sqrt{xy}) = (x^{1/p} + y^{1/p})^p,\, x,\,y > 0, \quad \varphi_p(x,\,y,\,0) = \begin{cases} x + y, \quad p > 0 \\ 0, \quad p < 0. \end{cases}$$
Let $f$ and $g$ be nonnegative functions on an arbitrary measure space $(X,\, \mathcal{A},\,\mu)$ with $\|f\|_p^p,\,\|g\|_p^p < \infty$. Note that the triple
$(\|f\|_p^p,\, \|g\|_p^p,\, \|fg\|_{p/2}^{p/2}) \in \Omega$
by the Cauchy-Schwarz inequality. By the equality case, if the triple is in $\partial \Omega$ we have $\|f + g\|_p^p = \varphi_p(\|f\|_p^p,\, \|g\|_p^p,\, \|fg\|_{p/2}^{p/2})$.
Our approach is based on the following lemma:
\begin{lemma}\label{GeneralConcavity}
Let $p \in \mathbb{R} \backslash \{0\}$, and assume that $H \in C(\Omega)$ is a concave, one-homogeneous function on $\Omega$ with $H|_{\partial \Omega} = \varphi_p$. Then
$$\|f + g\|_p^p \leq H\left(\|f\|_p^p,\, \|g\|_p^p,\, \|fg\|_{p/2}^{p/2}\right).$$
If $H$ is convex, the inequality reverses.
\end{lemma}
\begin{proof}
By the boundary conditions, we have
$$1 = H\left(\frac{f^p}{(f+g)^p},\, \frac{g^p}{(f+g)^p},\, \frac{(fg)^{p/2}}{(f+g)^p}\right)$$
on the set $X' = \{f + g > 0\}$ when $p > 0$, or $\{f + g < \infty\}$ when $p < 0$. Integrating this identity with respect to the probability measure $\frac{(f+g)^p\,d\mu}{\|f+g\|_p^p}$ on $X'$ and applying Jensen's inequality gives
$$1 \leq H\left(\frac{\|f\|_p^p}{\|f+g\|_p^p},\, \frac{\|g\|_p^p}{\|f+g\|_p^p}, \, \frac{\|fg\|_{p/2}^{p/2}}{\|f+g\|_p^p}\right)$$
when $H$ is concave, and the other inequality for $H$ convex. The result follows from the one-homogeneity of $H$.
\end{proof}
Lemma \ref{GeneralConcavity} reduces our problem to computing the concave and convex envelopes of $\varphi_p$ on $\Omega$. By concave envelope we mean the infimum of linear functions on $\Omega$ 
that are greater than $\varphi_p$ on $\partial \Omega$, and by convex envelope the supremum of linear functions on $\Omega$ that are smaller than $\varphi_p$ on $\partial \Omega$. 
Let $\overline{H}_p$ denote the concave envelope, and $\underline{H}_p$ the convex envelope. For $(x,\,y,\,z) \in \Omega$, define
$$w(x,\,y,\,z) := \frac{2z}{x+y}, \quad v(x,\,y,\,z) := \min\left\{\frac{x}{z},\, \frac{y}{z},\,1\right\},$$
where we take $w = 0$ at the origin and $v = 1$ on $\Omega \cap \{z = 0\}$. Define the one-homogeneous functions $F_p,\,G_p$ on $\Omega$ by
\begin{align}
F_p(x,\,y,\,z) := \frac{x+y}{2}\,\left((1+\sqrt{1-w^2})^{1/p} + (1-\sqrt{1-w^2})^{1/p}\right)^p, \label{env11}\\
G_p(x,\,y,\,z) := \begin{cases} x + y + \left((v^{1/p} + v^{-1/p})^p - (v + v^{-1})\right)\,z, \quad p > 0 \\ (v^{1/p} + v^{-1/p})^p \,z, \quad p < 0. \end{cases}\label{env22}
\end{align}

\begin{proposition}\label{EnvelopeFormulae}
The concave and convex envelopes $\overline{H}_p,\, \underline{H}_p$ of $\varphi_p$ in $\Omega$ are in $C(\Omega)$ and are given explicitly by the formulae
$$\overline{H}_p = \begin{cases}
F_p, \quad p \in (0,\,1] \cup [2,\,\infty), \\
G_p, \quad p \in (-\infty,\,0) \cup (1,\,2)
\end{cases}$$
and
$$\underline{H}_p = \begin{cases}
F_p, \quad p \in (-\infty,\,0) \cup (1,\,2), \\
G_p, \quad p \in (0,\,1] \cup [2,\,\infty).
\end{cases}$$
\end{proposition}
\noindent We delay the proof of Proposition \ref{EnvelopeFormulae} to Section \ref{EnvelopeComputations}, and immediately note that Theorems \ref{teor1} and \ref{teor2} follow quickly:

\begin{proof}[{\bf Proof of Theorems \ref{teor1} and \ref{teor2}:}]
To prove the inequalities, just apply Lemma \ref{GeneralConcavity} to the functions $\overline{H}_p$ and $\underline{H}_p$.
To check the equality cases, observe that in the proof of Lemma \ref{GeneralConcavity}, we have equality in Jensen provided $\{(f^p,\,g^p,\,(fg)^{p/2})\}$ lie in a set where $H$ is linear.

Since $F_p$ is linear when restricted to the hyperplanes $\{z = k(x+y)\} \cap \Omega$ (which are nontrivial when $k \in [0,\,1/2]$) we obtain the equality case in Theorem \ref{teor1}.

We note that $G_p$ is linear on the triangular cone $\{z \leq \min\{x,\,y\}\} \cap \Omega$, and on the hyperplanes $\{z = \gamma x\} \cap \Omega$ and $\{z = \gamma y\} \cap \Omega$ for each $\gamma \geq 1$.
The first condition gives $(fg)^{p/2} \leq \min\{f^p,\,g^p\}$, so $f = g$ on $\{fg > 0\}$ in the case $p > 0$ and on $\{fg < \infty\}$ in the case $p < 0$.
The second condition gives $(fg)^{p/2} = \gamma f^{p}$, and the third $(fg)^{p/2} = \gamma g^{p}$. When $p > 0$ the second condition gives that $g = \lambda f$ on $\{f > 0\}$ for some $\lambda \geq 1$,
and the third gives that $f = \lambda g$ on $\{g > 0\}$ for some $\lambda \geq 1$; when $p < 0$ the second condition gives $g = \lambda f$ on $\{f < \infty\}$ for some $\lambda \leq 1$,
and the third gives that $f = \lambda g$ on $\{g < \infty\}$ for some $\lambda \leq 1$.
\end{proof}

To conclude the section we address the optimality of Theorems \ref{teor1} and \ref{teor2} in the measure space $(X,\, \mathcal{A},\,\mu) = ([0,\,1],\, \mathcal{B},\, dx)$.
We define
$$\overline{B}_p(x,\,y,\,z) = \sup\left\{ \|f+g\|_p^p : (\|f\|_p^p,\,\|g\|_p^p,\,\|fg\|_{p/2}^{p/2}) = (x,\,y,\,z)\right\},$$
$$\underline{B}_p(x,\,y,\,z) = \inf\left\{ \|f+g\|_p^p : (\|f\|_p^p,\,\|g\|_p^p,\,\|fg\|_{p/2}^{p/2}) = (x,\,y,\,z)\right\}.$$
It is easy to see that $\overline{B}_p,\,\underline{B}_p$ are defined on a cone $\Omega_p \subset \Omega$, are locally bounded by the inequalities $(f + g)^p \leq 2^{p-1}(f^p + g^p)$ for $p \in (-\infty,\,0) \cup [1,\,\infty)$
and $(f+g)^p \leq f^p + g^p$ for $p \in (0,\,1]$, are one-homogeneous, and equal $\varphi_p$ on $\partial\Omega$ (by the
equality case of Cauchy-Schwarz). Furthermore, by Lemma \ref{GeneralConcavity} we have
$$\underline{H}_p \leq \underline{B}_p \leq \overline{B}_p \leq \overline{H}_p$$
on the common domain of definition.

\begin{lemma}\label{EnvelopeEquality}
If $\overline{B}_p\,(\underline{B}_p)$ is defined on all of $\Omega$ and is concave (convex), then
$$\overline{H}_p = \overline{B}_p \quad (\underline{B}_p = \underline{H}_p).$$
\end{lemma}
\begin{proof}
Local boundedness and concavity of $\overline{B}_p$ implies continuity in the interior of $\Omega$, and since $\overline{B}_p$ is
trapped between envelopes that attain the data continuously, we have $\overline{B}_p \in C(\Omega)$. Since $\overline{H}_p$ is the 
smallest such concave function, we conclude that $\overline{B}_p \geq \overline{H}_p$. The argument is similar for $\underline{B}_p$.
\end{proof}

Thus, it just remains to show that when $(X,\,\mathcal{A},\,\mu) = ([0,\,1],\, \mathcal{B},\,dx)$, the domain of definition for $\overline{B}_p$ and $\underline{B}_p$ is all of $\Omega$,
and that $\overline{B}_p$ is concave and $\underline{B}_p$ is convex.

\begin{lemma}\label{EnvelopeLemma}
For $(X,\,\mathcal{A},\,\mu) = ([0,\,1],\, \mathcal{B},\,dx)$ we have $\Omega_p = \Omega$ for all $p \neq 0$, that $\overline{B}_p$ is concave in $\Omega$,
and that $\underline{B}_p$ is convex in $\Omega$.
\end{lemma}

\noindent The optimality of the inequalities in Theorems \ref{teor1} and \ref{teor2} follows:

\begin{proof}[{\bf Proof of Optimality Statements:}]
For either inequality, given 
$$(x,\,y,\,z) = (\|f\|_p^p,\, \|g\|_p^p,\, \|fg\|_{p/2}^{p/2}),$$ 
the functions $\overline{B}_p(x,\,y,\,z)$ and $\underline{B}_p(x,\,y,\,z)$ are by definition the best we can do. 
These are equal to the envelopes $\overline{H}_p,\,\underline{H}_p$ by Lemmas \ref{EnvelopeEquality} and \ref{EnvelopeLemma}.
\end{proof}

\begin{remark}
For given $(x,\,y,\,z) \in \Omega$ and $p \in \mathbb{R} \backslash \{0\}$, the supremum (infimum) in the definition of $\overline{B}_p \, (\underline{B}_p)$ is in fact attained. 

For equality in (\ref{pir1}) consider pairs of the form $(f,\,g) = (a,\,b)\chi_{[0,\,c]} + (b,\,a)\chi_{[c,\,1]}$ for $a,\,b,\,c$ chosen appropriately. 

For equality in (\ref{pir2}), consider pairs of the form 
$$
(f,\,g) = (a,\,a) \chi_{[0,\,1/2]} + (b,\,0)\chi_{[1/2,\,3/4]} + (0,\, c)\chi_{[3/4,\,1]}
$$
 for $a,\,b,\,c$ appropriately chosen when $z \leq \min\{x,\,y\}$,
and $(f,\,g) = (a,\,b)\chi_{[0,\,1/2]} + (c,\,d)\chi_{[1/2,\,1]}$ when $z > \min\{x,\,y\}$ for appropriate $a,\,b,\,c,\,d$, with one of $c,\,d$ equal to $0$. 

For equality in (\ref{pir3}), consider pairs of the form 
$$
(f,\,g) = (a,\,a) \chi_{[0,\,1/2]} + (b,\,\infty)\chi_{[1/2,\,3/4]} + (\infty,\, c)\chi_{[3/4,\,1]}
$$
 for $a,\,b,\,c$ appropriately chosen when $z \leq \min\{x,\,y\}$,
and $(f,\,g) = (a,\,b)\chi_{[0,\,1/2]} + (c,\,d)\chi_{[1/2,\,1]}$ when $z > \min\{x,\,y\}$ for appropriate $a,\,b,\,c,\,d$, with one of $c,\,d$ equal to $\infty$.
\end{remark}

\begin{proof}[{\bf Proof of Lemma \ref{EnvelopeLemma}:}]
For the first part, if $p > 0$ take $f_s = (2x)^{1/p}\chi_{[s,\,s+1/2]}$ for $s \in [0,\,1/2]$ and let $g = (2y)^{1/p}\chi_{[1/2,\,1]}$. Then $\|f_s\|_p^p = x$ and $\|g\|_p^p = y$. Furthermore, 
we have $h(s) := \|f_sg\|_{p/2}^{p/2}$ is continuous, increasing, and $h(0) = 0,\, h(1/2) = \sqrt{xy}$. When $p < 0$, use the same example but set $f_s,\,g = \infty$
where they were previously zero.

For the second part, let $(x_i,\,y_i,\,z_i) \in \Omega$ with $i = 1,\,2$, and for $\epsilon > 0$ choose $f_i,\,g_i$ such that
$(x_i,\,y_i,\,z_i) = (\|f_i\|_p^p,\, \|g_i\|_p^p,\, \|f_ig_i\|_{p/2}^{p/2})$ and
$$\|f_i + g_i\|_p^p \geq \overline{B}_p(x_i,\,y_i,\,z_i) - \epsilon, \quad i =  1,\,2.$$
Extend $f_i,\,g_i$ to be zero outside of $[0,\,1]$, and define the rescalings
$$\tilde{f_1}(s) = 2^{1/p} f_1(2s),\, \tilde{g_1}(s) = 2^{1/p}g_1(2s),\, \tilde{f_2}(s) = 2^{1/p}f_2(2s-1),\, \tilde{g_2}(s) = 2^{1/p}g_2(2s-1),$$
so that $\tilde{f}_i,\,\tilde{g}_i$ are supported in $[0,\,1/2]$ for $i = 1$ and in $[1/2,\,1]$ for $i = 2$.
We then have
\begin{align*}
\frac{\overline{B}_p(x_1,\,y_1,\,z_1) + \overline{B}_p(x_2,\,y_2,\,z_2)}{2} - \epsilon &\leq \frac{1}{2}\left(\|\tilde{f}_1 + \tilde{g}_1\|_{L^p([0,\,1/2])}^p + \|\tilde{f}_2 + \tilde{g}_2\|_{L^p([1/2,\,1])}^p\right) \\
&= \frac{1}{2}\|\tilde{f}_1 + \tilde{g}_1 + \tilde{f}_2 + \tilde{g}_2\|_p^p \\
&= \left\|\frac{\tilde{f}_1 + \tilde{f}_2}{2^{1/p}} + \frac{\tilde{g}_1 + \tilde{g}_2}{2^{1/p}}\right\|_p^p \\
&\leq \overline{B}_p\left(\frac{1}{2}(x_1 + x_2,\, y_1 + y_2,\, z_1 + z_2)\right).
\end{align*}
For the last inequality, we used that for $f_0 := 2^{-1/p}(\tilde{f}_1 + \tilde{f}_2),\, g_0 := 2^{-1/p}(\tilde{g}_1 + \tilde{g}_2)$ we have
$$\|f_0\|_p^p = \frac{1}{2}(x_1+x_2),\, \|g_0\|_p^p = \frac{1}{2} (y_1 + y_2), \, \|f_0g_0\|_{p/2}^{p/2} = \frac{1}{2}(z_1 + z_2).$$
Taking $\epsilon \rightarrow 0$, we conclude that $\overline{B}_p$ is concave. The convex direction is similar.
\end{proof}

\begin{remark}
Lemma \ref{EnvelopeLemma} holds for any measure space with translation and scaling properties similar to $([0,\,1],\, \mathcal{B},\,dx)$, e.g. $(B_1 \subset \mathbb{R}^n,\, \mathcal{B},\, dx)$.
\end{remark}

\begin{remark}
The fact that $\overline{B}_p$ is concave also follows from Theorem $1$ in \cite{Paata}. Since the argument is simple, we decided to include it for the reader's convenience.
\end{remark}

\section{Envelopes}\label{EnvelopeComputations}

In this section we prove Proposition \ref{EnvelopeFormulae}. We begin with some simple observations.

First, to check concavity (convexity) in $\Omega$ and continuity up to $\partial \Omega$ of $\overline{H}_p \, (\underline{H}_p)$, by one-homogeneity it suffices to check these properties
on the half-ellipse
$$D := \Omega \cap \{x+y = 2\}.$$ 
More generally, any one-homogeneous function $B$ in a convex cone in $\mathbb{R}^n$ (say contained in $\{x_n > 0\}$) is concave (convex) if it is concave (convex)
when restricted to a cross-section of the cone (say $\{x_n = 1\}$). Indeed, by one-homogeneity we have 
$$B\left(\frac{x+y}{2}\right) = \frac{x_n + y_n}{2}B\left(\lambda\frac{x}{x_n} + (1-\lambda) \frac{y}{y_n}\right)$$
where $\lambda = \frac{x_n}{x_n + y_n}$, and the statement follows by applying concavity / convexity of $B$ on the cross-section and then using one-homogeneity once more.

Second, to prove that $\overline{H}_p \, (\underline{H}_p)$ is the concave (convex) envelope of $\varphi_p$, it suffices to check that each point in the interior of $D$ lies on a segment that connects boundary points of $D$,
on which $\overline{H}_p \,(\underline{H}_p)$ is linear. Indeed, then any linear function larger (smaller) than $\varphi_p$ on $\partial \Omega$ will then be larger than $\overline{H}_p$
(smaller than $\underline{H}_p$) in the interior of $\Omega$.

\begin{proof}[{\bf Proof of Proposition \ref{EnvelopeFormulae}}]
We first examine $F_p$, and then $G_p$.

\vspace{3mm}

{\bf The Function $F_p$.} On $D$ we can write $F_p(1+s,\,1-s,\,t) = u(t),$ where
$$u(t) := \left[ (1+\sqrt{1-t^{2}})^{1/p}+(1-\sqrt{1-t^{2}})^{1/p} \right]^{p}, \quad t\in [0,1].$$ 
It is clear that $F_p$ is continuous up to $\partial D$ for each $p \in \mathbb{R} \backslash \{0\}$, and $u(0) = \varphi_p$ (that is, $2$ if $p > 0$ and $0$ if $p < 0$) on the bottom of $D$ and
$F_p(1-s,\,1+s,\,\sqrt{1-s^2}) = ((1+s)^{1/p} + (1-s)^{1/p})^p = \varphi_p$
on the top of $D$. Since $F_p$ is constant along the horizontal segments in $D$, it suffices to check that $u$ is concave when $p \in (0,\,1] \cap [2,\,\infty)$, and convex otherwise. 
To that end, we let $t =\sin(x)$, with $x \in [0,\pi/2]$. Then 
\begin{align*}
u(\sin(x)) = \left[ (1+\cos(x))^{1/p}+(1-\cos(x))^{1/p} \right]^{p}.
\end{align*}
Let us rewrite the last equality as  
\begin{align*}
\frac{1}{2} u(\sin(2s)) =  \left[\sin^{2/p}(s)+\cos^{2/p}(s) \right]^{p},
\end{align*}
where $s = x/2 \in [0,\,\pi/4]$. Differentiating both sides of the equality in $s$, we obtain
\begin{align*}
&u'(\sin(2s)) \cos(2s) \\
&= p \left[ \sin ^{2/p}(s)+\cos^{2/p}(s) \right]^{p -1} \frac{2}{p}\left(\sin^{2/p -1}(s) \cos(s) - \cos^{2/p-1}(s) \sin(s) \right)\\
&= p\left[ \sin ^{2/p}(s)+\cos^{2/p}(s) \right]^{p -1} \frac{2 \cos^{2/p}(s)}{p}\left(\tan^{2/p -1}(s) - \tan(s) \right).
\end{align*}
Taking the derivative a second time we obtain
\begin{align*}
&2u''(\sin(2s))\cos^{2}(2s)  - 2u'(\sin(2s))\sin(2s) = \\
&p(p-1)\left[ \sin ^{2/p}(s)+\cos^{2/p}(s) \right]^{p -2} \left[\frac{2 \cos^{2/p}(s)}{p}\left(\tan^{2/p -1}(s) - \tan(s) \right) \right]^{2}+\\
&p\left[ \sin ^{2/p}(s)+\cos^{2/p}(s) \right]^{p -1} \times \left( - \frac{4 \cos^{2/p}(s) \tan(s)}{p^{2}}\left(\tan^{2/p -1}(s) - \tan(s) \right) +\right.\\
&\left.  \frac{2 \cos^{2/p}(s)}{p}\left( \left( \frac{2}{p}-1\right)\tan^{2/p -2}(s) - 1\right)\left(1+\tan^{2}(s)\right) \right).
\end{align*}

Therefore 
\begin{align*}
&2u''(\sin(2s)) \cos^{2}(2s) = \left[ \sin ^{2/p}(s)+\cos^{2/p}(s) \right]^{p -2} \times  \frac{4}{p} \times \cos^{4/p}(s)\times  \\
&\left[ (p-1) \left[\left(\tan^{2/p -1}(s) - \tan(s) \right) \right]^{2}+\left[ 1+\tan^{2/p}(s) \right] \times \right. \\
&\left( - \tan(s)\left(\tan^{2/p -1}(s) - \tan(s) \right) + \left( \left(1- \frac{p}{2}\right)\tan^{2/p -2}(s) - \frac{p}{2}\right)\left(1+\tan^{2}(s)\right) \right)+\\
&\left.p\tan(2s) \left[ 1+\tan^{2/p}(s) \right] \left(\tan^{2/p -1}(s) - \tan(s) \right)\right].
\end{align*}

Since $\tan(2s) = \frac{2\tan(s)}{1-\tan^{2}(s)}$, after denoting   $\tan(s) = w \in [0,1]$ we obtain

\begin{align*}
&\frac{2u''(\sin(2s)) \cos^{2}(2s)}{\left[ \sin ^{2/p}(s)+\cos^{2/p}(s) \right]^{p -2}  \cos^{4/p}(s)} \\
&= \frac{4(p-1)}{p} \left(w^{2/p -1} - w \right)^{2} \\
&+\frac{4( 1+w^{2/p})}{p} \times \left( - w^{2/p} + w^{2} + \left( \left(1- \frac{p}{2}\right)w^{2/p -2} - \frac{p}{2}\right)\left(1+w^{2}\right) \right)\\
&+\frac{8w}{1-w^{2}} (1+w^{2/p}) (w^{2/p -1} - w) \\
&= \frac{2(1+w^{2})^{2}}{1-w^{2}} \left( w^{\frac{4}{p} -2}+\left(\frac{2}{p} -1\right)w^{\frac{2}{p} -2}(1-w^{2})-1\right).
\end{align*}
(The last equality is a tedious computation, but can be checked by hand). Since $\frac{2(1+w^{2})^{2}}{1-w^{2}}>0$ we see after denoting $x:=w^{2} \in [0,1]$ that $\sign(u'') = \sign(v(x))$, where 
\begin{align*}
v(x) = x^{\frac{2}{p}-1} + \left(\frac{2}{p}-1\right)x^{\frac{1}{p}-1} (1-x)-1, \quad x \in [0,1].
\end{align*}
Let us study the sign of $v(x)$. Without loss of generality assume that $p\neq 1, 2$, otherwise the claims about concavity/convexity of $u$ are trivial.  First notice that $v(1)=0$, and 
$$v'(x) = x^{\frac{1}{p}-2}\left(\frac{2}{p}-1\right)\left(x^{\frac{1}{p}} - \left(1+\frac{1}{p}(x-1) \right) \right).$$
Therefore, if $p \in (2,\infty)$  it follows from concavity of $x \mapsto x^{1/p}$ that $v'\geq 0$, and hence $v \leq 0$, i.e., $u$ is concave. Similarly, if $ p \in (1,2)$, then $v \geq 0$, i.e., $u$ is convex. Next, if $p \in (0,1)$ then $x \mapsto  x^{1/p}$ is convex, and hence $v'\geq 0$, i.e., $u$ is concave. Finally, if $p \in (-\infty,0)$ then $x \mapsto x^{1/p}$ is convex, and therefore $v' \leq 0$, i.e., $u$ is convex. 

\vspace{3mm}


{\bf The Function $G_p$.} Let $b_p(s,\,z) = G_p(1+s,\,1-s,\,z)$, with $(s,\,z)$ in the upper half-disc. For $p > 0$ we can write $b_p$ explicitly as
$$b_p(s,\,z) =  2 + \begin{cases} w(1-|s|,\,z), \quad z \geq 1-|s| \\ (2^p-2)z, \quad z < 1-|s|, \end{cases}$$
where $w$ is the one-homogeneous function given by
$$w(t,\,z) := \left(t^{1/p} + (z^2/t)^{1/p}\right)^p - (t + (z^2/t))$$
with $(t,\,z) \in (0,\,1)^2$. It is easy to check that $b_p$ continuously takes the boundary values $b_p(s,\,0) = 2 = \varphi_p$ and $b_p(s,\,\sqrt{1-s^2}) = ((1+s)^{1/p} + (1-s)^{1/p})^p = \varphi_p$. Let 
$$h(t) := w(t,\,1) = \left(t^{1/p} + t^{-1/p}\right)^p - (t + t^{-1}), \quad t \in (0,\,1).$$
By the one-homogeneity of $w$ and the fact that $b_p$ is linear on the triangle $\{z < 1-|s|\}$
with vertical gradient, if we show that $h'(1) = 0$ and that $h$ is concave / convex on $[0,\,1]$, then $b_p$ is $C^1$ away from $(s,\,z) = (\pm 1,\,0)$ and concave / convex.
Furthermore, $b_p$ is linear when restricted to the segments through $(s,\,z) = (\pm 1,\,0)$ that lie outside of the triangle $\{z \leq 1-|s|\}$, so $G_p$ is the concave / convex envelope 
provided the above conditions on $h$ are confirmed. To that end we compute the first two derivatives of $h$. The first derivative is
$$h'(t) = (t^{1/p} + t^{-1/p})^{p-1}(t^{1/p-1} - t^{-1/p-1}) - (1-t^{-2}).$$
This confirms that $h'(1) = 0$. The second derivative is
\begin{align*}
h''(t) &= \frac{p-1}{p}(t^{1/p} + t^{-1/p})^{p-2}(t^{1/p-1} - t^{-1/p-1})^2 \\
&+ \frac{1}{p}(t^{1/p} + t^{-1/p})^{p-1}((1-p)t^{1/p-2} + (1+p)t^{-1/p-2}) - 2t^{-3} \\
&= \frac{1}{p}(t^{1/p} + t^{-1/p})^{p-2}[(p-1)(t^{1/p-1} - t^{-1/p-1})^2  \\
&+ (t^{1/p} + t^{-1/p})((1-p)t^{1/p-2} + (1+p)t^{-1/p-2})] - 2t^{-3} \\
&= \frac{2}{p}(t^{1/p} + t^{-1/p})^{p-2}[pt^{-2/p-2} + (2-p)t^{-2}] - 2t^{-3} \\
&= 2t^{-3}[(t^{1/p} + t^{-1/p})^{p-2}(t^{1-2/p} + (2/p-1) t) - 1] \\
&= 2t^{-3}[(1 + t^{2/p})^{p-2}(1 + (2/p-1)t^{2/p}) - 1].
\end{align*}
Let $x := t^{2/p} \in [0,\,1]$. It suffices to show that
$$g_p(x) := (1 + (2/p - 1)x) - (1+x)^{2-p}$$
satisfies $g_p \leq 0$ on $[0,\,1]$ for $p \in (1,\,2)$ and $g_p \geq 0$ on $[0,\,1]$ for $p \in (0,\,1] \cup [2,\,\infty)$.
Note that $g_p(0) = 0$. The desired inequality for $g_p(1)$ is equivalent
to the fact that the linear function $p$ crosses the convex function $2^{p-1}$ at $p = 1$ and $p = 2$. Finally, we observe that the first term
in $g_p$ is linear, and the second term is convex for $p \in (1,\,2)$ and concave for $p \in (0,\,1) \cup (2,\,\infty)$. The desired inequality for
$g_p(x)$ with $x \in (0,\,1)$ follows immediately from this observation and the inequalities at the endpoints $x = 0$ and $x = 1$.

When $p < 0$ we can write $b_p$ explicitly as
$$b_p(s,\,z) =  \begin{cases} \tilde{w}(1-|s|,\,z), \quad z \geq 1-|s| \\ 2^p\,z, \quad z < 1-|s|, \end{cases}$$
where $\tilde{w}$ is the one-homogeneous function given by
$$\tilde{w}(t,\,z) := \left(t^{1/p} + (z^2/t)^{1/p}\right)^p$$
with $(t,\,z) \in (0,\,1)^2$. The same considerations as above reduce the problem to showing that
$$\tilde{h}(t) := \tilde{w}(t,\,1) = \left(t^{1/p} + t^{-1/p}\right)^p$$
satisfies $\tilde{h}'(1) = 0$ and $\tilde{h}$ is concave on $[0,\,1]$. We have
$$\tilde{h}' = (t^{1/p} + t^{-1/p})^{p-1}(t^{1/p-1} - t^{-1/p-1}) \Rightarrow \tilde{h}'(1) = 0,$$ 
$$\tilde{h}'' = 2t^{-2}(t^{1/p} + t^{-1/p})^{p-2}[t^{-2/p} + (2/p-1)],$$
and the conclusion follows quickly using $p < 0$.
\end{proof}

\begin{remark}\label{TriangleSupport}
It follows from the concavity / convexity properties of $G_p$ that
$$G_p(x,\,y,\,z) \leq x + y + (2^p - 2)z$$
when $p \in [1,2]$, and the inequality reverses for $p \in (0,\,1] \cup [2,\,\infty)$. Indeed, $G_p$ agrees with the linear function on the right hand side on an open set. 
We conclude from Theorem \ref{teor2} that for any nonnegative numbers $a,b$, and any $p \in [1,2]$, we have
\begin{align*}
(a+b)^{p} \leq a^{p}+b^{p}+(2^{p}-2)(ab)^{p/2},
\end{align*}
and the inequality reverses if $p \in (0,1] \cup [2,\infty)$.
\end{remark}

\section{Proof of Corollary~\ref{cor2}}\label{Corollary2}
In this final section we prove Corollary \ref{cor2}.

\begin{proof}[{\bf Proof of Corollary \ref{cor2}:}]
Recall from Remark \ref{TriangleSupport} that for any nonnegative numbers $a,b$, and any $p \in [1,2]$, we have
\begin{align*}
(a+b)^{p} \leq a^{p}+b^{p}+(2^{p}-2)(ab)^{p/2},
\end{align*}
and the inequality reverses for $p \in (0,\,1] \cup [2,\,\infty)$. Since for $p \in [0,2]$ we have $(a+b)^{p/2}\leq a^{p/2}+b^{p/2}$, and the
reverse inequality if $p\geq 2$, it follows by induction that for any nonnegative numbers $a_{j}\geq 0$ we have 
\begin{align}\label{pin}
(\sum a_{j})^{p} \leq \sum_{j} a_{j}^{p} + (2^{p}-2)\sum_{i<j} (a_{i}a_{j})^{p/2}
\end{align}
holds true for $p\in [1,2]$, and the reverse inequality if $p \in (0,1]\cup [2, \infty)$. 
Finally it remains to put $a_{j}=f_{j}(x)$ and integrate the inequality. 
\end{proof}

\begin{remark}
When $p < 0$, inequality (\ref{pin}) does not hold with three or more $a_j$. Take e.g.  $a_j = 1$ for $j \leq 3$.
\end{remark}

\section{Concluding Remarks on Envelopes}

An important challenge in this work was to compute the envelopes (\ref{env11}) and (\ref{env22}). In this section we briefly explain how we found them.

We recall from Section \ref{Reductions} that for the measure space $([0,1], \mathcal{B}, dx)$ we have $\overline{B}_p = \overline{H}_p$ is defined on $\Omega$, one-homogeneous, and equals $\varphi_p$
on $\partial \Omega$; that is, $\overline{H}_p(x,\,y,\,\sqrt{xy}) = (x^{1/p} + y^{1/p})^p$. We also recall from the discussion at the beginning of Section \ref{EnvelopeComputations} that by one-homogeneity, to compute $\overline{H}_p$ it is enough to restrict our attention to the 
cross-section $D = \Omega \cap \{x+y=2\}$. Writing $D = \{(1+s,\,1-s,\,z)\}$ with $(s,\,z)$ in the upper half-disc, this reduces the problem understanding how the upper boundary of the convex envelope of the space curve 
$$\gamma(s) = (s, \sqrt{1-s^{2}}, ((1-s)^{1/p}+(1+s)^{1/p})^{p}), \quad s \in [-1,1] $$
looks. One can show that the torsion $\tau_{\gamma}$ of the space curve $\gamma$ changes sign only once from $-$ to $+$, at $s=0$, when $p \in (0,1) \cup (2, \infty)$, and from $+$ to $-$  when $p \in (-\infty,0) \cup (1, 2)$. Consider the case $p \in (0,1) \cup (2, \infty)$. Then it follows from Lemma~29 of Section~3.2 in \cite{Paata2} that locally, say for some $\delta \in (0,1]$,  there exists a function $a(s) : [0, \delta] \to [-1,0]$ such that $a(0)=0$, $a(s)$ is strictly decreasing, and the function $B(u,w)$ defined parametrically by
 \begin{align*}
 &B(\lambda (a(s), \sqrt{1-a(s)^{2}}) + (1-\lambda) (s, \sqrt{1-s^{2}}))=\\
 & \lambda((1-a(s)^{1/p}+(1+a(s))^{1/p})^{p} + (1-\lambda)((1-s)^{1/p}+(1+s)^{1/p})^{p}
 \end{align*} 
for $ \lambda \in [0,1],  s \in [0, \delta]$ is concave. In other words $B$ has the prescribed boundary condition, i.e., 
$B(s, \sqrt{1-s^{2}}) = ((1-s)^{1/p}+(1+s)^{1/p})^{p}$, it is linear along the line segments 
$\ell(s):=[(a(s), \sqrt{1-a(s)^{2}}), (s, \sqrt{1-s^{2}})],$
 and $B$ is concave. It follows that ``locally'' $B$ is a concave envelope. Because of the symmetry in $x$ and $y$ of the boundary data $\varphi_p$, one can show that the line segments $\ell(s)$ must be horizontal, i.e., 
 $a(s) = -s$, and in fact $\delta =1$. This means that $B$ is a global concave envelope
 \begin{align*}
 B(u,w)=((1-\sqrt{1-w^{2}})^{1/p}+(1+\sqrt{1-w^{2}})^{1/p})^{p}
 \end{align*}
 for all $|u|\leq 1$ and $0\leq w \leq \sqrt{1-u^{2}}$. Now it remains to change variables back to recover the envelope (\ref{env11}). 
 
 The case $p \in (-\infty,0) \cup (1, 2)$ is different because $\tau_{\gamma}$ changes sign from $+$ to $-$, and in this case an ``angle'' arises with vertex sitting around the point $s=0$ (see Section~3 in \cite{Paata2}).



\end{document}